\documentclass[reqno,10pt]{amsart}
\usepackage{amscd,amssymb,amsmath,color}
\usepackage{latexsym}
\usepackage[all]{xy}
\usepackage[colorlinks=true]{hyperref}
\usepackage{defs}

\def\tf{{\rm tf}}
\def\card{{\rm card}}
\def\End{{\rm End}}

\def\whLa{{\wh{L^a}}}
\def\tilalp{{\wt{\alp}}}
\def\spl{{\rm spl}}
\def\topdim{{\rm top.dim}}
\def\Topdim{{\rm Top.dim}}
\def\topdeg{{\rm top.tr.deg}}
\def\Topdeg{{\rm Top.tr.deg}}
\def\rmlog{{\rm log}}

\def\whKa{{\wh{K^a}}}

\def\whOmega{{\wh\Omega}}

\def\whd{{\wh{d}}}

\begin{document}

\author{Michael Temkin}
\title{Topological transcendence degree}
\thanks{This research was supported by the Israel Science Foundation (grant No. 1159/15). I would like to thank the referee for pointing out numerous inaccuracies in the first version of the paper.}
\address{Einstein Institute of Mathematics, The Hebrew University of Jerusalem, Giv'at Ram, Jerusalem, 91904, Israel}
\email{temkin@math.huji.ac.il}
\begin{abstract}
Throughout the paper, an {\em analytic} field means a non-archime\-dean complete real-valued field, and our main objective is to extend the basic theory of transcendental extensions to these fields. One easily introduces a topological analogue of the transcendence degree, but, surprisingly, it turns out that it may behave very badly. For example, a particular case of a theorem of Matignon-Reversat, \cite[Th\`eor\'eme~2]{matignon-reversat}, asserts that if $\cha(k)>0$ then $\wh{k((t))^a}$ possesses non-invertible continuous $k$-endomorphisms, and this implies that the topological transcendence degree is not additive in towers. Nevertheless, we prove that in some aspects the topological transcendence degree behaves reasonably, and we show by explicit counter-examples that our positive results are pretty sharp. Applications to types of points in Berkovich spaces and untilts of $\wh{\bfF_p((t))^a}$ are discussed.
\end{abstract}

\maketitle

\section{Introduction}

\subsection{Some history}
To some extent, a general theory of analytic fields was developed by Kaplansky and his school. This includes, for example, the theory of spherically complete fields, but not the questions we consider in this paper. Below we briefly discuss some cases where analytic fields show up in other areas of mathematics.

In the classical arithmetics and algebraic geometry, analytic fields usually showed up as completions of algebraic extensions of discretely valued fields. The class of analytic fields obtained this way is rather narrow, and this explains why the theory of their extensions was mainly developed for algebraic extensions (e.g., ramification theory) and their completions (e.g., Ax-Sen theorem).

Non-archimedean geometry naturally produces a much larger class of analytic fields and raises various questions about extensions of such fields, including those not of topologically algebraic type. In fact it seems, that most results about such extensions were established for the sake of applications to non-archimedean geometry. This started already in the framework of rigid geometry: the stability theorem, see \cite[Theorem~5.3.2/1]{bgr}, and the thesis of M. Matignon, where he studied one-dimensional transcendental extensions. Unfortunately, some results of Matignon were not known to the experts in other fields, including the theorem of Matignon-Reversat, \cite[Th\`eor\'eme~2]{matignon-reversat}.

In the framework of Berkovich non-archimedean geometry, the relation to analytic fields becomes even more natural, as the completed residue field $\calH(x)$ is the most important invariant associated to a point $x$ in a Berkovich space. In particular, the author extended the theory of one-dimensional extensions in \cite[\S6]{temst} and \cite[\S3]{insepunif}, and deduced some applications to analytic and birational geometries. In algebraic (or henselian) setting, similar questions were also studied by F.-V. Kuhlmann, e.g. see \cite{Kuhlmann-elimination1}.

Finally, analytic fields also appear as complete residue fields of analytic points on adic and perfectoid spaces, and this leads to similar questions.

\subsection{Motivation}

\subsubsection{Foundations}
Although the abstract theory of valued fields has been studied for a long period, various natural questions about topologically transcendental extensions of analytic fields were not touched. The primary goal of this paper is to fill in this gap in foundations. The original motivation for the author was the problem of introducing types of points on Berkovich spaces. A final push for writing the paper was a question on classification of endomorphisms of $\wh{\bfF_p((t))^a}$.

\subsubsection{Untilts of $\wh{\bfF_p((t))^a}$}
Fargues and Fontaine have recently defined a ``complete curve" $X$ which sheds a new light on p-adic Hodge theory. The residue fields of $X$ are parameterized by analytic fields $K$ of mixed characteristic such that their tilt $K^\flat$ is isomorphic to $\wh{\bfF_p((t))^a}$. It was then a natural question whether one necessarily has that $\bfC_p\toisom K$, see \cite[Remark~2.24]{fargues-fontaine-durham}. In fact, several experts expected the answer to be positive, but a counter-example was constructed in \cite{untilt}. The idea is very simple: $K$ is automatically algebraically closed, and tilting the canonical embedding $\phi\:\bfC_p\into K$ one obtains an endomorphism $\phi^\flat$ of $\wh{\bfF_p((t))^a}$. Conversely, via the tilting correspondence one can untilt any such endomorphism $\phi_1^\flat$ to an embedding $\phi_1\:\bfC_p\into K$, hence the above question reduces to the question whether any endomorphism $\phi_1^\flat$ is invertible. Thus, the theorem of Matignon-Reversat answers both questions negatively.

The above result indicates that, probably, one should give up with classifying points on $X$. Nevertheless, we prove in Theorem~\ref{subfieldsth} that all analytic algebraically closed subfields of $\wh{\bfF_p((t))^a}$ are totally ordered by inclusion. In particular, even if a full classification is difficult or impossible, it seems plausible that there might be valuable numerical invariants of non-trivial extensions $\wh{\bfF_p((x))^a}/\wh{\bfF_p((t))^a}$.

The theorem about subfields is equivalent to the claim that any pair of elements in $\wh{\bfF_p((t))^a}$ is topologically algebraically dependent. Such result lays the ground for a basic theory of extensions of topological transcendence degree 2. In this paper, we prove an analogous result (see Theorem~\ref{mainineq}) for arbitrary extensions.

\subsubsection{Multitype of points on Berkovich spaces}\label{multisec}
Berkovich divides points on $k$-analytic curves into 4 types, see \cite[\S1.4.4]{berbook}. Let us denote it as a quadruple $(n_1,n_2,n_3,n_4)$ of three zeros and a unit. Given a point $x$ on a $d$-dimensional space $X_d$, one is tempted to give the following recursive definition: locally at $x$ find a map $f\:X_d\to X_{d-1}$ such that $\dim(X_{d-1})=d-1$ and the fiber near $x$ is a curve, and define the naive multitype of $x$ in $X$ to be the naive multitype of $f(x)$ plus the type of $x$ in the fiber. Question: is this well-defined?

It is easy to see that the entries $n_2$ and $n_3$ are well-defined, and in fact
$$n_2(x)=\trdeg(\wHx/\tilk),\ \ n_3(x)=\dim_\bfQ((|\calH(x)^\times|/|k^\times|)\otimes\bfQ).$$ However, if $\cha(\tilk)>0$ then the theorem of Matignon-Reversat allows to construct points for which $n_1$ and $n_4$ depend on choices, see \S\ref{multitype}. On the positive side, the theory of topological transcendence degree works perfectly well when $\cha(\tilk)=0$, and it follows that in this case the naive multitype is well-defined and, in fact, $n_2+n_3+n_4=\topdeg(\calH(x)/k)$. This leads to the correct definition:

\begin{defin}\label{multirem}
The multitype of $x$ consists of four entries with $n_2(x)$ and $n_3(x)$ defined above and
$$n_1(x)=d-\topdeg(\calH(x)/k), \ \ \ n_4(x)=\topdeg(\calH(x)/k)-n_2(x)-n_3(x).$$
\end{defin}

We will prove in Theorem~\ref{multith} that $n_4$ is the minimal number of type 4 points in the fibers one can obtain for a sequence $X_d\to X_{d-1}\to X_{d-2}\to\dots$.

\subsection{The results and an overview of the paper}

\subsubsection{The main inequality}
In \S\ref{topgebraicsec}, we introduce a basic terminology. In particular, we define the {\em topgebraic} (i.e. topologically algebraic) independence of a set and define the topological transcendence degree $\topdeg(K/k)$ to be the minimal bound on the cardinality of an independent set. Also, we define $\Topdeg(K/k)$ to be the minimal cardinality of a topgebraically generating set. In \S\ref{mainsec}, we use a simple perturbation argument to prove that $\topdeg(K/k)\le\Topdeg(K/k)$, see Theorem~\ref{mainineq}. This is the necessary minimum for the notion of topological transcendence degree to make any practical sense, and this is the only not completely trivial result that holds true for arbitrary extensions of analytic fields. If $\Topdeg(K/k)<\infty$ then both cardinals are equal by Theorem~\ref{basisth}, but they may differ in general.

\subsubsection{Completed differentials}
In \S\ref{diffsec}, we introduce the strongest technical tool we use to study transcendental extensions: the completed module of differentials. In the classical theory of transcendental field extensions the module $\Omega_{K/k}$ completely controls transcendental properties of the extension when $\cha(k)=0$. For example, $S\subset K$ is algebraically independent over $k$ (resp. is a transcendence basis) if and only if $d_{K/k}(S)$ is linearly independent over $K$ (resp. is a basis), and hence $\trdeg(K/k)=\dim_K(\Omega_{K/k})$. This fails in positive characteristic since given a tower $K/l/k$ for which $K/l$ is inseparable, the map $\psi_{K/l/k}\:\Omega_{l/k}\otimes_lK\to\Omega_{K/k}$ is not injective. For example, $d_{K/k}$ vanishes on $K^p$.

The situation in the analytic case is analogous. We prove in \S\ref{char0sec} that if the residue characteristic of $k$ is zero then the maps $\hatpsi_{K/l/k}\:\hatOmega_{l/k}\wtimes_lK\to\hatOmega_{K/k}$ are always injective and $\topdeg(K/k)$ equals the topological dimension of $\hatOmega_{K/k}$ (see \S\ref{topdimsec}). Taking into account that $\hatOmega_{K/k}$ is a Banach space rather than an abstract vector space, the theory is as nice as possible.

We know by the theorem of Matignon-Reversat that a similar theory does not work properly when $p=\cha(\tilk)>0$, but completed differentials allow to pin down the main source of the problem: the map $\hatpsi_{K/l/k}$ does not have to be injective. Furthermore, we prove in Lemma~\ref{noninjlem} that if this happens then either $K/l$ is not separable, as in the classical case of field extensions, or $\Omega_{\Kcirc/\lcirc}$ contains a non-trivial infinitely divisible torsion element. In particular, $K/l$ has zero different and is ``very wildly ramified". This case can be interpreted as a topological analogue of inseparability.

\subsubsection{Pathologies}
Construction of various pathological counter-examples is postponed until the last section. This is done because we use a little bit of material about perturbations and completed differentials, but this section is relatively self-contained. Therefore, we advise the reader to look through \S\ref{pathologysec} directly after reading \S\ref{topgebraicsec} and before reading the rest of the paper.

Roughly speaking, there are three main pathologies dealt with in the three subsections of \S\ref{pathologysec}. First, a primitive extension $\wh{k(t)}$ of $k$ may contain an infinite algebraic extension $l/k$. For our applications the main case of interest is when $t\notin\whka$ and $l/k$ has zero different, but for the sake of completeness we study other cases in detail too, see \S\ref{primsec}.

The main pathology is constructed in Theorem~\ref{mainpathology}: a tower of topologically transcendental extensions $L=\wh{k(x,y)}/K=\wh{k(x)}/k$ such that $x\in\wh{k(y)^a}$. In particular, the topological transcendence degree of all three extensions $L/K$, $L/k$ and $K/k$ is one. The main idea of the construction of $L/K/k$ is to achieve that $\whd_{K/k}(x)$ dies in $\hatOmega_{L/k}$, and in spirit of Lemma~\ref{noninjlem}, the main thing one should achieve is that $L/K$ has zero different. This is achieved by a direct computation in \cite{untilt}, while in the current paper we use Theorem~\ref{deepramlem} instead. We use Theorem~\ref{mainpathology} to provide a new short proof of the theorem of Matignon-Reversat, but it is also easy to deduce our Theorem~\ref{mainpathology} from the theorem of Matignon-Reversat, see Theorem~\ref{endth} and Remark~\ref{endrem2}.

Finally, it may happen that $\topdeg(K/k)<\Topdeg(K/k)$. In such case one always has that $\Topdeg(K/k)=\infty$, and such large extensions rarely arise in applications (with the main exception being extensions $K/k$ with a spherically complete $K$). Nevertheless, we think that it is important to know about their existence, and using the theorem of Matignon-Reversat it is easy to construct such an extension even with $\topdeg(K/k)=1$, see Theorem~\ref{nobasisth}. Finally, we show in Theorem~\ref{sphcor} that infinite extensions $K/k$ with an algebraically closed and spherically complete $K$ are also very large in the following sense: at least if $\cha(\tilk)>0$ then $K/k$ does not admit a topological transcendence basis.

\section{Topgebraic dependence}\label{topgebraicsec}

\subsection{Terminology}

\subsubsection{Real-valued and analytic fields}
By a {\em real-valued field} we always mean a field $K$ provided with a non-archimedean real valuation $|\ |\:K\to\bfR_+$. Then $\Kcirc$, $\Kcirccirc$ and $\tilK$ denote the ring of integers of $K$, the maximal ideal of $\Kcirc$, and the residue field of $K$, respectively. For any extension or embedding of real-valued fields we always assume that the valuations agree. We will mainly work with {\em analytic fields}, which are complete real-valued fields.

\subsubsection{Extension of valuation}
Any analytic field $K$ is henselian. Thus, if $L$ is an algebraic extension of $K$ then the valuation of $K$ extends to $L$ uniquely and we denote the corresponding completion by $\hatL$.

\subsubsection{Topological generation}
Let $K/k$ be an extension of analytic fields and $S\subseteq K$ a subset. Then the closure $l$ of $k(S)$ in $K$ is an analytic field which is also the completion of $k(S)$ provided with the valuation induced from $K$. So, we will use the notation $l=\wh{k(S)}$ and say that $l$ is {\em topologically generated} by $S$ over $k$. We say, that $K$ is {\em topologically finitely generated} over $k$ if $K=\wh{k(S)}$ for a finite set $S$.

\subsubsection{Topgebraic extensions}
In this paper the combination ``topologically algebraic" appears very often so we abbreviate it as {\em topgebraic}. In particular, a completed algebraic closure $\whKa$ will be called a {\em topgebraic closure} of $K$. More generally, an extension of analytic fields $L/K$ is called {\em topgebraic} if $\wh{K^aL}=\whKa$. This happens if and only if $L/K$ is an analytic subextension of $\whKa/K$. In general, we say that an element $t\in K$ is {\em topgebraic} over $k$ if so is the extension $\wh{k(t)}/k$, and we call $t$ {\em topologically transcendental} over $k$ otherwise.

\subsubsection{Primitive extensions}
We say that $K/k$ is {\em primitive} if $K=\wh{k(t)}$ for a single element $t$. By Abhyankar inequality, $E_{k(t)/k}+F_{k(t)/k}\le 1$, where $F_{k(t)/k}=\trdeg(\wt{k(t)}/\tilk)$ and $E_{K/k}=\dim_\bfQ(|K^\times|/|k^\times|\otimes\bfQ)$. Since these invariants do not change under completion we also have that $E_{k(t)/k}+F_{k(t)/k}\le 1$. Note that $K/k$ is primitive if and only if $K=\calH(x)$ for a point $x$ on Berkovich affine line over $k$, and we classify primitive extensions similarly to the classification of points on $\bfA^1_k$ in \cite[\S1.4.4]{berbook}:
\begin{itemize}
\item[(1)] Type 1 if $K/k$ is topgebraic.
\item[(2)] Type 2 if $F_{K/k}=1$. (Happens if and only if $\tilK/\tilk$ is transcendental.)
\item[(3)] Type 3 if $E_{K/k}=1$. (Happens if and only if $|K^\times|/|k^\times|$ is not torsion.)
\item[(4)] Type 4 otherwise.
\end{itemize}

\subsubsection{Topological transcendence degree}
A set $S\subset L$ is {\em topgebraically independent} over $K$ if any element $x\in S$ is topologically transcendental over $\wh{K(S\setminus\{x\})}$. The {\em topological transcendence degree} of $L$ over $K$ is the minimal cardinal $\Lam$ such that the size of any topgebraically independent set does not exceed $\Lam$. We use the notation $\topdeg(L/K)=\Lam$.

\begin{rem}
We do not study the question wether maximal topgebraically independent sets exist and are of cardinality $\topdeg(L/K)$. The problem with existence is that the union of an increasing sequence of topgebraically independent sets does not have to be topgebraically independent.
\end{rem}

\subsubsection{Topological transcendence basis}
We say that $S\subset L$ {\em topgebraically generates} $L$ over $K$ if $L$ is topgebraic over $\wh{K(S)}$. If, in addition, $S$ is topgebraically independent then we say that $S$ is a {\em topological transcendence basis} of $L$ over $K$.

\subsubsection{Topgebraic generating degree}
The set of cardinalities of all topgebraically generating subsets of $K$ has a minimal element that we call the {\em topgebraic generating degree} and denote $\Topdeg(K/k)$.

\begin{rem}\label{topgenrem}
(i) If $\Topdeg(K/k)<\infty$ then $K/k$ possesses a topological transcendence basis. Indeed, take any finite topgebraically generating set $S$ and remove elements until $S$ becomes minimal. Then $S$ also becomes topgebraically independent.

(ii) We stress that even for $\wh{k(t)^a}/k$ there might exist maximal topgebraically independent sets which are not topgebraically generating, see Remark~\ref{endrem2}(i)
.

(iii) One of various pathologies with very large extensions is that minimal topgebraically generating sets do not have to exist in general. In particular, a topological transcendence basis does not have to exist. For example, we will construct in Theorem~\ref{nobasisth} an extension with $\topdeg(K/k)=1$ and $\Topdeg(K/k)=\infty$.
\end{rem}


\subsection{Simple properties}

\subsubsection{Monotonicity}
We say that a set-theoretic invariant $\phi$ of field extensions is {\em monotonic} if $\phi(L'/K')\le\phi(L/K)$ for any tower $L/L'/K'/K$.

\begin{lem}\label{monotonlem}
The topological transcendence degree is a monotonic invariant.
\end{lem}
\begin{proof}
If a set $S\subset L'$ is topgebraically independent over $K'$ then it is also topgebraically independent over $K$.
\end{proof}

\begin{rem}
It is not clear if $\Topdeg$ is monotonic. Questions~\ref{topdeg1q} and \ref{subfieldsq} might be related to this problem, especially in order to construct non-monotonic examples (if exist).
\end{rem}

\subsubsection{Subadditivity}
We say that an invariant $\phi$ of field extensions is {\em subadditive} if $\phi(F/K)\le\phi(F/L)+\phi(L/K)$ for any tower $F/L/K$.

\begin{lem}\label{subadditivelem}
The topological generating degree is a subadditive invariant.
\end{lem}
\begin{proof}
If $S$ generates $L/K$ and $T$ generates $F/L$ then $S\cup T$ generates $F/K$.
\end{proof}

\begin{question}
Is $\topdeg$ subadditive?
\end{question}

\section{The main inequality}\label{mainsec}
The aim of this section is to show that $\topdeg(K/k)\le\Topdeg(K/k)$.

\subsection{Perturbations}

\subsubsection{Radius}
Let $L/K$ be an extension of analytic fields and let $t\in L$ be an element. By the {\em $K$-radius} of $t$ we mean the number $r_K(t)=\inf_{c\in K^a}|t-c|$, where the absolute value is computed in $\whLa$. In particular, $r_K(t)=0$ if and only if $t$ is topgebraic over $K$.

\begin{rem}
The geometric meaning of the radius is that $t$ induces a morphism $\calM(L)\to\bfA^1_K$ whose image is a point $x_t\in\bfA^1_K$ and $r_K(t)$ is the radius of $x_t$.
\end{rem}

\subsubsection{Perturbations of fields}
Let $\phi,\psi\:K\into L$ be two embeddings of analytic fields. We say that $\psi$ is  a {\em quasi-perturbation} of $\phi$ if $|\phi(x)-\psi(x)|<|x|$ for any $x\in K$. Furthermore, we say that $\psi$ is a {\em perturbation} of $\phi$ if this inequality holds uniformly in the following sense: there exists a number $\alp<1$ such that
$|\phi(x)-\psi(x)|<\alp|x|$ for any $x\in K$. To specify $\alp$, we will also say that $\psi$ is an {\em $\alp$-perturbation} of $\phi$.

\begin{rem}\label{deformrem}
(i) Perhaps the main property of perturbations is the following simple result: $[L:\phi(K)]=[L:\psi(K)]$, whenever $\phi$ is a perturbation of $\psi$, see \cite[Lemma~6.3.3]{temst}. In particular, if an endomorphism $\phi\:K\to K$ is a perturbation of the identity then $\phi$ is an automorphism.

(ii) The notion of quasi-perturbations is much less useful. In particular, \cite[Question~1.1]{untilt} asks when an extension of analytic fields $K/k$ admits non-invertible quasi-perturbations of $\Id_K$ trivial on $k$. One such example is given in \cite[Example~3.2]{untilt}; it uses $K=\wh{k(t_1,t_2,\dots)}$ and applies to any $k$ with a non-discrete group of values. A more surprising example is when $K=\wh{k(x)^a}$ and $\cha(\tilk)>0$. It exists by a theorem of Matignon-Reversat, see \cite[Th\`eor\'eme~2]{matignon-reversat}, and see also \cite[Theorem~1.2]{untilt} for a shorter construction. We will return to this topic in \S\ref{endsec}.
\end{rem}

\subsubsection{Primitive extensions}
Primitive extensions of type 1 are much more pathological than the topologically transcendental ones. The main reason for this is that they are ``rigid", as opposed to the non-topgebraic extensions that admit a lot of small perturbations by \cite[Lemma~6.3.2]{temst}. The latter lemma will be heavily used below, so we recall it for convenience of the reader.

\begin{lem}\label{primdeform}
Assume that $L/K$ is an extension of analytic fields and $t,t'\in L$ are elements such that $|t-t'|<r_K(t)$. Then there is a unique isomorphism of analytic fields $\phi\:\wh{k(t)}\toisom\wh{k(t')}$ such that $\phi(t)=t'$. In addition, $\phi$ is an $\alp$-perturbation of $\Id_{\wh{k(t)}}$, where $|t-t'|/r_K(t)<\alp<1$.
\end{lem}

\subsubsection{Multiradius}
Assume now that $S$ is a subset of $L$. Then we define the {\em $K$-multiradius} of $S$ to be the map $r_{K,S}\:S\to[0,\infty)$ given by $r_{K,S}(t)=r_{\wh{K(S\setminus\{t\}})}(t)$. In particular, the set $S$ is topgebraically independent over $K$ if and only if $r_{K,S}(t)>0$ for any $t\in S$.

\subsubsection{Perturbations of subsets}
Let $L/K$ be an analytic extension and $S\subset L$ a subset. By an $\alp$-perturbation of $S$ in $L$ we mean a map $\phi\:S\to L$ such that $|s-\phi(s)|<\alp r_{K,S}(s)$ for any $s\in S$. Note that if $S$ admits an $\alp$-perturbation then $r_{K,S}(t)>0$ for any $t\in S$ and hence $S$ is topgebraically independent over $K$. Furthermore, since $|t-t'|\ge r_{K,S}(t)$ for any two distinct elements $t,t'\in S$, the map $\phi$ is injective and is uniquely determined by the set $S'=\phi(S)$. For this reason, we will freely say in the sequel that $S'$ is an $\alp$-perturbation of $S$.

\begin{theor}\label{deformth}
Assume that $L/K$ is an extension of analytic fields, $\alp<1$ a number and $S\subset L$ is a subset with an $\alp$-perturbation $S'\subset L$. Then $\phi\:S\toisom S'$ uniquely extends to an isomorphism of analytic fields $\phi\:\wh{K(S)}\toisom\wh{K(S')}$. In addition, $\phi$ is an $\alp$-perturbation of $\Id_{\wh{K(S)}}$.
\end{theor}
\begin{proof}
For any $T\subseteq S$ define $\phi_T\:S\to L$ by $\phi_T(x)=\phi(x)$ if $x\in T$ and $\phi_T(x)=x$ otherwise. There is at most one extension of $\phi_T$ to a homomorphism of analytic fields $\wh{K(S)}\toisom\wh{K(\phi_T(S))}$. If exists, the latter is automatically an isomorphism, and we will also denote it by $\phi_T$. Let $C$ denote the set of subsets $T\subseteq S$ such that $\phi_T\:\wh{K(S)}\toisom\wh{K(\phi_T(S))}$ exists and is an $\alp$-perturbation of $\Id_{\wh{K(S)}}$. We should prove that $S\in C$.

First, we claim that if $T\in C$ and $x\in S\setminus T$ then $T'=T\cup\{x\}$ is also in $C$. Indeed, $\phi_{T'}=\phi_T\circ\phi_{\{x\}}$ so this follows from the fact that $\phi_T$ is an $\alp$-pertubation by the assumption and $\phi_{\{x\}}$ is an $\alp$-perturbation by Lemma~\ref{primdeform} applied to $\wh{K(S\setminus\{x\})}$ and the elements $x$ and $\phi(x)$.

Next, we claim that if $\{T_i\}$ is an ordered chain in $C$ when $T=\cup T_i$ is also in $C$. Once this is proved, the theorem will follow by Zorn's lemma. On the level of maps $S\to L$, we have that $\phi_T$ is the limit of the maps $\phi_{T_i}$. It follows that for any $x\in K(S)$ the sequence $\phi_{T_i}(x)$ stabilizes for $i$ large enough, and setting $\phi_T(x)$ to be this limit value we obtain an isomorphism of valued fields $\phi_T\:K(S)\toisom K(\phi_T(S))$, which is an $\alp$-perturbation of the identity. Completing both sides we obtain a required isomorphism of analytic fields $\phi_T\:\wh{K(S)}\toisom\wh{K(\phi_T(S))}$, and it remains to show that the latter is also an $\alp$-perturbation of the identity. Given a non-zero element $\hatx\in\wh{K(S)}$ choose $x\in K(S)$ such that $|x-\hatx|<\alp|x|$. Then $|\phi_T(\hatx)-\phi_T(x)|<\alp|x|$ since $\phi_T$ preserves the valuation, and $|x-\phi_T(x)|<\alp|x|$ since $\phi_T|_{K(S)}$ is an $\alp$-perturbation of the identity. The three inequalities imply that  $|\hatx-\phi_T(\hatx)|<\alp|x|=\alp|\hatx|$, as required.
\end{proof}

\begin{cor}
Keep the assumptions of Theorem~\ref{deformth}, then $r_{K,S'}\circ\phi=r_{K,S}$. In particular, being an $\alp$-perturbation is an equivalence relation on the set of subsets of $L$.
\end{cor}

\subsection{The inequality}
Using perturbations we can now compare topgebraically independent and topgebraically generating subsets.

\begin{theor}\label{mainineq}
Let $L/K$ be an extension of analytic fields. If $S\subset L$ is topgebraically independent over $K$ and $T\subseteq L$ is topgebraically generating then $|S|\le|T|$. In particular, $\topdeg(L/K)\le\Topdeg(L/K)$.
\end{theor}
\begin{proof}
Replacing $L$ by $\wh{L^a}$ we can assume that $L=\wh{K(T)^a}$. Then $K(T)^a$ is dense in $L$, and so for any $\alp\in(0,1)$ there exists an $\alp$-perturbation $S'$ of $S$ such that $S'\subset K(T)^a$. Since $S'$ is topgebraically independent over $K$, it is also algebraically independent and hence $|S|=|S'|\le |T|$.
\end{proof}

\subsubsection{Finite topgebraic generation}
If $\Topdeg(K/k)<\infty$ one can say much more.

\begin{theor}\label{basisth}
If $L/K$ is of finite topgebraic degree then $\topdeg(L/K)=\Topdeg(L/K)$, $L$ possesses a topological transcendence basis, and the cardinality of any such base is $\topdeg(L/K)$.
\end{theor}
\begin{proof}
By Remark~\ref{topgenrem}(i), $L/K$ possesses a topological transcendence basis $S$. Therefore, $\topdeg(L/K)\ge\Topdeg(L/K)$ and in view of Theorem~\ref{mainineq} we obtain an actual equality. In addition, Theorem~\ref{mainineq} implies that all such bases are of the same cardinality.
\end{proof}

\section{Completed differentials}\label{diffsec}
In this section, we will use completed modules of differentials to study topological transcendence degree.

\subsection{Basic facts}
Basic facts about differentials of valuation rings can be found in \cite[Chapter~6]{Gabber-Ramero} and \cite[Section~4]{Temkintopforms}. The latter also studies completed modules of differentials. First, we briefly recall what will be needed.

\subsubsection{Differentials}
We will use the notation $d_{B/A}\:B\to\Omega_{B/A}$ and $d_B\:B\to\Omega_B$ to denote the differentials.

\subsubsection{Completed differentials}
For any homomorphism $A\to K$ with $K$ a real-valued field, the module $\Omega_{K/A}$ possesses a natural seminorm called K\"ahler seminorm, and the completion is denoted $\whOmega_{K/A}$, see \cite[\S4.1.1 and \S4.3.1]{Temkintopforms}. We will denote the differential by $\whd_{K/A}\:K\to\whOmega_{K/A}$. Of a special importance will be the case when $\whd_{L/K}(t)\neq 0$ but $\whd_{F/K}(t)=0$ for a tower of analytic fields $F/L/K$.

\subsubsection{The first fundamental sequence}
As in \cite[Section~5]{Temkintopforms}, the first fundamental sequence plays a critical role in studying completed differentials.

\begin{lem}\label{firstseq}
Let $F/L/K$ be a tower of analytic fields, then there is a semi-exact sequence $$\whOmega_{L/K}\wtimes_LF\stackrel{\hatpsi_{F/L/K}}{\longrightarrow}\whOmega_{F/K}\to\whOmega_{F/L}\to 0.$$
\end{lem}
\begin{proof}
Complete the exact sequence from \cite[Lemma~4.2.2(i)]{Temkintopforms}. (Alternatively, one can deduce this directly from \cite[Lemma~4.3.3]{Temkintopforms}.)
\end{proof}

\subsubsection{Differentials of primitive extensions}\label{difsec}
We refer to \cite[Lemma~2.4]{untilt} for a simple direct computation of the completed module of differentials of a primitive extension $L=\wh{K(t)}$. The upshot is that $\whOmega_{L/K}$ is generated by $\whd_{L/K}(t)$, and it vanishes if and only if $K^s\cap L$ is dense in $L$. In particular, if it vanishes then the extension is of type 1.

\begin{rem}
The completion homomorphism $\Omega_{L/K}\to\whOmega_{L/K}$ may have a huge kernel. So, the K\"ahler seminorm on $\Omega_{L/K}$ is very far from being a norm.
\end{rem}

\subsubsection{Topological dependence}\label{topdimsec}
We say that a subset $T$ of an $L$-Banach space $V$ is {\em topologically generating} over $L$ if $V$ is the closure of the $L$-span of $T$. We say that $T$ is {\em topologically independent} if each proper subset of $T$ topologically generates a proper subspace of the subspace topologically generated by $T$ itself. The cardinals that bound topologically independent and generating sets from above and below, respectively, will be denoted $\topdim_K(V)$ and $\Topdim_K(V)$.

If $S$ is independent than any small enough perturbation of $S$ is also independent. More precisely, one can perturb each $v\in S$ by an element $\veps_v$ whose norm is strictly smaller than the distance from $v$ to the Banach space generated by the rest of $S$, and this is proved by Zorn's lemma precisely as its field analogue Theorem~\ref{deformth}. In particular, if $S$ is topologically independent and $T$ is topologically generating then we can find a perturbation $S'\subseteq\Span_L(T)$ of $S$ which is still independent. So $|S|=|S'|\le|T|$, and hence $\topdim_K(V)\le\Topdim_K(V)$. (Probably, the equality always holds, see Question~\ref{char0quest}(ii) below.)

\begin{lem}\label{toplinear}
Assume that $L/K$ is an extension of analytic fields and $S\subseteq L$ is a subset. Then

(i) If $S$ topologically generates $L$ then $\hatd_{L/K}(S)$ topologically generates $\hatOmega_{L/K}$.

(ii) If $\hatd_{L/K}(S)$ is topologically independent then $S$ is topgebraically independent. In particular, $\topdeg(K/k)\le\topdim_K(\hatOmega_{L/K})$.

(iii) Assume that $\cha(K)=0$. If $S$ is topgebraically generating then $\hatd_{L/K}(S)$ is topologically generating. In particular, $\Topdeg(K/k)\ge\Topdim_K(\hatOmega_{L/K})$.
\end{lem}
\begin{proof}
(i) Set $l=K(S)$, then $L=\hatl$ and hence $\hatOmega_{l/K}=\hatOmega_{L/K}$ by \cite[Corollary~5.6.7]{Temkintopforms}. Since $d_{l/K}(S)$ spans $\Omega_{l/K}$, its image $\hatd_{L/K}(S)$ topologically generates $\hatOmega_{L/K}$.

(ii) Set $F=\wh{K(S)}$. Choose any $x\in S$ and set $T=S\setminus\{x\}$ and $E=\wh{K(T)}$. Since $\whd_{E/K}(T)$ topologically generates $\whOmega_{E/K}$ by (i) and $x$ is not in the completed span of $\whd_{F/K}(T)$, it follows from the semi-exactness of the first fundamental sequence of $F/E/K$ that $\whd_{F/E}(x)\neq 0$. Hence $x$ is topologically transcendental over $E$ by \S\ref{difsec}. This proves that $S$ is topgebraically independent.

(iii) We know by (i) that $\hatd_{F/K}(S)$ topologically generates $\whOmega_{F/K}$. Using the first fundamental sequence for $L/F/K$ we see that it suffices to show that $\whOmega_{L/F}=0$. By Ax-Sen theorem, $l=L\cap F^s$ is dense in $L$. Clearly, $\Omega_{l/F}=0$ and by \cite[Corollary~5.6.7]{Temkintopforms} we obtain that $\whOmega_{L/F}=\whOmega_{l/F}=0$.
\end{proof}

Note that an analogue of (iii) does not hold in positive characteristic because already a finite inseparable extension has a non-trivial differential.

\subsubsection{$\Kcirc$-modules: completions and divisible elements}
We provide $\Kcirc$-modules $M$ with the $\pi$-adic topology, where $\pi=0$ if the valuation is trivial, and $\pi\in\Kcirccirc\setminus\{0\}$ otherwise. In particular, the kernel of the ($\pi$-adic) completion homomorphism $M\to\hatM$ consists of all {\em infinitely divisible} elements $x\in M$, i.e. $x=0$ if the valuation is trivial, and $x$ is divisible by any non-zero element of $\Kcirccirc$ otherwise.

\subsubsection{Differentials of rings of integers}
Given a homomorphism $f\:A\to K$ set $\Acirc=f^{-1}(\Kcirc)$. We make the following assumption on $f$, which will always be satisfied in the sequel: $A$ is a localization of $\Acirc$. Then $\Omega_{\Kcirc/\Acirc}\otimes_{\Kcirc}K=\Omega_{K/\Acirc}=\Omega_{K/A},$ so the maximal torsion-free quotient $(\Omega_{\Kcirc/\Acirc})_\tf$ of $\Omega_{\Kcirc/\Acirc}$ is a {\em semilattice} in $\Omega_{K/A}$, i.e. a $\Kcirc$-submodule of $\Omega_{K/A}$ that spans it as a $K$-vector space.

\begin{lem}\label{kahlerlem}
With the above notation, the K\"ahler seminorm of $\Omega_{K/A}$ is equivalent to the seminorm induced by $(\Omega_{\Kcirc/\Acirc})_\tf$. In particular, an element $x\in(\Omega_{\Kcirc/\Acirc})_\tf$ is infinitely divisible if and only if it is sent to zero in $\whOmega_{K/A}$.
\end{lem}
\begin{proof}
The first claim follows from the following two facts: (1) $\|\ \|_\Omega$ is the seminorm associated with a module $(\Omega_{\Kcirc/\Acirc}^\rmlog)_\tf$ by \cite[Theorem~5.1.8(ii)]{Temkintopforms},\footnote{I am grateful to the referee for pointing out that \cite[Theorem~5.1.8(ii)]{Temkintopforms} is stated incorrectly. This should be fixed by adding the assumption that $A$ is a localization of $\Acirc$ and does not affect anything else in that paper.} (2) the map $f\:\Omega_{\Kcirc/\Acirc}\to\Omega_{\Kcirc/\Acirc}^\rmlog$ is an isomorphism if the valuation is trivial, and the kernel and the cokernel of $f$ are killed by any element of $\Kcirccirc$ otherwise, see \cite[Corollary~5.3.3]{Temkintopforms}. In particular, it follows that $x$ is infinitely divisible if and only if $\|x\|_\Omega=0$, i.e. $x$ is killed by the completion homomorphism.
\end{proof}

\subsubsection{Non-injectivity of $\hatpsi_{F/L/K}$}
Now we can study when the first fundamental sequence does not extend to a short exact sequence.

\begin{lem}\label{noninjlem}
Let $F/L/K$ be a tower of analytic fields such that $\whOmega_{L/K}$ is finite-dimensional and $\hatpsi_{F/L/K}$ is not injective. Then either $F/L$ is not separable or $\Omega_{\Fcirc/\Lcirc}$ contains an infinitely divisible non-zero torsion element.
\end{lem}
\begin{proof}
It suffices, assuming that $F/L$ is separable, to find a non-zero infinitely divisible torsion element. The finite dimensionality implies that $\Omega_{L/K}\to\whOmega_{L/K}$ is onto. Lifting an element of $\Ker(\hatpsi_{F/L/K})$ to $\Omega_{L/K}\otimes_LF$ we obtain an element $x$ whose image in $\whOmega_{L/K}\wtimes_LF$ does not vanish but whose image in $\whOmega_{F/K}$ dies. Multiplying $x$ by an appropriate non-zero element of $\Lcirccirc$ we can also achieve that $x$ lies in $(\Omega_{\Lcirc/\Kcirc}\otimes_{\Lcirc}\Fcirc)_\tf$ and hence lifts to a non-torsion element $\xcirc\in\Omega_{\Lcirc/\Kcirc}\otimes_{\Lcirc}\Fcirc$. By \cite[Theorem~5.2.3(ii)]{Temkintopforms}, the map $$\psi^\circ_{\Fcirc/\Lcirc/\Kcirc}\:\Omega_{\Lcirc/\Kcirc}\otimes_{\Lcirc}\Fcirc\to\Omega_{\Fcirc/\Kcirc}$$ is injective, hence $y=\psi^\circ_{\Fcirc/\Lcirc/\Kcirc}(\xcirc)$ is a non-torsion element of $\Omega_{\Fcirc/\Kcirc}$ whose image in $\whOmega_{F/K}$ vanishes. By Lemma~\ref{kahlerlem}, we obtain that $\xcirc$ is not infinitely divisible but $y$ is. Find a non-zero $\pi\in\Lcirccirc$ such that $\pi^{-1}\xcirc\notin\Omega_{\Lcirc/\Kcirc}\otimes_{\Kcirc}\Lcirc$. Then the image of $\pi^{-1}y$ in $\Omega_{\Fcirc/\Lcirc}$ is an infinitely divisible non-zero torsion element.
\end{proof}

\subsection{Residue characteristic zero}\label{char0sec}
In this section we consider only analytic fields of residual characteristic zero. Using completed differentials we will show that in this case the topgebraic theory is nearly as nice as the algebraic one.

\subsubsection{The first fundamental sequence}
The crucial fact we are going to use is that the torsion of differentials is bounded when $\cha(\tilK)=0$.

\begin{lem}\label{torsionlem}
Let $L/K$ be an extension of valued fields of residual characteristic zero. Then the torsion of $\Omega_{\Lcirc/\Kcirc}$ is killed by any element of $\Kcirccirc$.
\end{lem}
\begin{proof}
By \cite[Theorem 5.2.3(ii)]{Temkintopforms}, $\Omega_{\Lcirc/\Kcirc}\otimes_{\Lcirc}(L^a)^\circ$ embeds into $\Omega_{(L^a)^\circ/\Kcirc}$, hence it suffices to prove the claim for $L^a/K$. Assume now that $L=L^a$. Then using the first fundamental sequence for $\Lcirc/(K^a)^\circ/\Kcirc$ and the fact that $\Omega_{\Lcirc/(K^a)^\circ}$ is torsion free by \cite[Theorem~6.5.20(i)]{Gabber-Ramero}, we see that it suffices to prove the theorem for the extension $K^a/K$. The latter case is covered by \cite[Lemma~5.2.7]{Temkintopforms} since $K^a/K$ is tame.
\end{proof}

In Lemma~\ref{noninjlem} we saw two situations when $\hatpsi_{F/L/K}$ may fail to be injective. Neither can happen when $\cha(\tilK)=0$. Indeed, $F/L$ is separable as $\cha(L)=0$ and $\Omega_{\Fcirc/\Lcirc}$ contains no infinitely divisible non-zero torsion elements by Lemma~\ref{torsionlem}. In fact, the situation with the first fundamental sequence is as good as possible when $\cha(\tilK)=0$.

\begin{theor}\label{exactadmiss}
If $F/L/K$ is a tower of analytic fields of residual characteristic zero then the sequence $$0\to\whOmega_{L/K}\wtimes_LF\stackrel{\hatpsi_{F/L/K}}{\longrightarrow}\whOmega_{F/K}\to\whOmega_{F/L}\to 0$$ is exact and admissible.
\end{theor}
\begin{proof}
Since $\cha(K)=0$, we have the short exact sequence $$0\to\Omega_{L/K}\otimes_LF\stackrel{\psi_{F/L/K}}{\longrightarrow}\Omega_{F/K}\stackrel{f}\to\Omega_{F/L}\to 0.$$ Once we prove that this exact sequence is admissible, the theorem will follow by applying the completion functor. By Lemma~\ref{kahlerlem}, the (K\"ahler) seminorm on $\Omega_{L/K}\otimes_LF$ is equivalent to the one induced by the semilattice $(\Omega_{\Lcirc/\Kcirc}\otimes_{\Lcirc}\Fcirc)_\tf$, and the seminorm on $\Ker(f)$ is equivalent to the one induced from the restriction of $(\Omega_{\Fcirc/\Kcirc})_\tf$ onto $\Ker(f)$. The quotient of these two semilattices embeds into $\Omega_{\Fcirc/\Lcirc}$ via the first fundamental sequence for $\Fcirc/\Lcirc/\Kcirc$, hence by Lemma~\ref{torsionlem} it is killed by elements of $\Lcirccirc$. So, the semilattices define equivalent seminorms, i.e. the sequence is admissible.
\end{proof}

\subsubsection{Linearization}
Now, we can strengthen Lemma~\ref{toplinear} as follows.

\begin{theor}\label{char0th}
Assume that $F/K$ is an extension of analytic fields of residual characteristic zero. Then a subset $S\subseteq F$ is a topgebraically generating (resp. independent) over $K$ if and only if the subset $\whd_{F/K}(S)\subseteq\whOmega_{F/K}$ is topologically generating (resp. independent) over $F$. In particular, $\topdeg(F/K)=\topdim_F(\hatOmega_{F/K})$ and $\Topdeg(F/K)=\Topdim_F(\hatOmega_{F/K})$.
\end{theor}
\begin{proof}
We start with the following claim corresponding to $S=\emptyset$: the extension $F/K$ is topgebraic if and only if $\whOmega_{F/K}=0$. Indeed, if $F\subseteq\whKa$ then $\whOmega_{F/K}\wtimes_K\whKa$ embeds into $\whOmega_{\whKa/K}$ by Theorem~\ref{exactadmiss}, but $\whOmega_{\whKa/K}=0$ because $\Omega_{K^a/K}=0$. Conversely, if $F/K$ is not topgebraic then we pick $t\in F$ which is not topgebraic over $K$ and set $L=\wh{K(t)}$. By \S\ref{difsec} $\whd_{L/K}(t)\neq 0$ and applying Theorem~\ref{exactadmiss} to $F/L/K$ we obtain that $\whd_{F/K}(t)\neq 0$, proving that $\whOmega_{F/K}\neq 0$.

Now, set $L=\wh{K(S)}$ and recall that $\whd_{L/K}(S)$ topologically generates $\whOmega_{L/K}$ by Lemma~\ref{toplinear}(i). Applying Theorem~\ref{exactadmiss} to $F/L/K$ we obtain that $\whd_{F/K}(S)$ is topologically generating if and only if $\whOmega_{F/L}$ vanishes. By the above claim this happens if and only if $F/L$ is topgebraic, that is $S$ topgebraically generates $F$ over $K$.

In the same way, applying Theorem~\ref{exactadmiss} to the tower $\wh{K(S)}/\wh{K(T)}/K$, where $T$ is a subset of $S$, one sees that $\wh{K(S)}$ is topgebraic over $\wh{K(T)}$ if and only if $\whd_{F/K}(T)$ and $\whd_{F/K}(S)$ topologically generate the same Banach subspace. This implies the claim about topgebraic independence.
\end{proof}

\subsubsection{Applications to topological transcendence degree}
Now, we can prove that, under a finiteness assumption, the topological transcendence degree behaves as nice as possible.

\begin{theor}\label{topdegchar0}
Let $F/K$ be an extension of analytic fields with $\cha(\tilK)=0$.

(i) If $\topdeg_{F/K}<\infty$ then $$\topdeg(F/K)=\Topdeg(F/K)=\dim_F(\whOmega_{F/K})$$ and this is also the size of any maximal topgebraically independent subset $S\subset F$ and any minimal topgebraically generating subset $T\subset F$. In particular, any such subset is a topological transcendence basis of $F/K$.

(ii) If $L$ is an intermediate analytic field and either $F/K$ or both $F/L$ and $L/K$ are of finite topological transcendence degree then $$\topdeg(F/K)=\topdeg(F/L)+\topdeg(L/K).$$
\end{theor}
\begin{proof}
Note that a finite subset of $\whOmega_{F/K}$ is linearly independent if and only if it is topologically linearly independent. Hence Theorem~\ref{char0th} reduces all assertions of (i) to basic claims of usual linear algebra. Part (ii) is proved similarly, but one should also use Theorem~\ref{exactadmiss}.
\end{proof}

\subsubsection{Infinite degree}
The situation with extensions of infinite degree is not so nice, but at least the complexity comes only from the usual theory of Banach spaces. We do not pursue this direction, but here are some speculations.

\begin{question}\label{char0quest}
(i) If $\cha(\tilK)=0$, can it happen that $F/K$ does not admit a topological transcendence basis? Here is an idea of the construction: find a $K$-Banach space $V$ that does not admit a topological basis (i.e. a topologically independent and generating set), provide $K[V]$ with the maximal norm extending the norms of $K$ and $V$, show that this norm is multiplicative and hence induces a valuation on $K(V)=\Frac(K[V])$, set $F=\wh{K(V)}$ and prove that $\whOmega_{F/K}=V$.

(ii) Is it true that nevertheless $\topdeg(F/K)=\Topdeg(F/K)$ for any extension $F/K$? Note that by Theorem~\ref{char0th} it suffices to show that $\topdim_K(V)=\Topdim_K(V)$ for the $K$-Banach space $V=\whOmega_{F/K}$. It seems plausible that this equality holds for an arbitrary $V$, and I am grateful to Andrzej Szankowski for suggesting to prove this using the topological dual $V'$ and a maximal biorthogonal family $\{S\subset V,S'\subset V'\}$. Such a family exists by Zorn's lemma, and, at least, this reduces the question to showing that $\Topdim_K(V')\ge\Topdim_K(V)$.
\end{question}

\section{Pathologies}\label{pathologysec}
In this section we collect various pathological examples that should illustrate the difference with the algebraic theory of field extensions.

\subsection{Primitive extensions}\label{primsec}
The main pathology revealed by primitive extensions $l/k$ is that for type 1 and 4 extensions it can happen that $l$ contains an infinite algebraic extension of $k$.

\subsubsection{Splitting radius}
For any $\alp\in k^a$ by the {\em splitting radius} $r_\spl(\alp/k)$ we mean the minimal distance between $\alp$ and its $k$-conjugates. In particular, $r_\spl(\alp/k)>0$ if and only if $\alp$ is separable over $k$. By a version of Krasner's lemma, see \cite[Lemma~3.1.3(i)]{insepunif}, $r<r_\spl(\alp/k)$ if and only if the disc $E_k(\alp,r)$ is a split $k(\alp)$-disc.

\subsubsection{Type 1 case}
We start with extensions of type 1. As we are going to prove, a surprisingly huge class of topgebraic extensions are primitive. In particular, this shows that being topologically finitely generated is a rather weak condition comparing to the algebraic analogue.

\begin{theor}\label{type1exam}
Let $k$ be an analytic field with a non-trivial valuation. If $l$ is the completion of a countably generated separable algebraic extension of $k$ then the extension $l/k$ is primitive.
\end{theor}
\begin{proof}
By our assumption $l=\wh{k(a_0,a_1,\dots)}$ for a sequence $a_0,a_1,\dots$ of elements of $k^s$. Choose non-zero elements $\pi_0,\pi_1,\dots$ in $k$ such that the sequence $|\pi_ia_i|$ monotonically tends to zero and $|\pi_na_n|<r_\spl(\pi_na_n/k_n)$ for each $n>0$, where $k_n=k(a_0\.a_{n-1})$ and $k_0=k$. Consider the element $t=\sum_{i=0}^\infty\pi_ia_i\in\whka$, then $l'=\wh{k(t)}$ is an analytic subfield of $l$. We claim that $a_n\in l'$ for any $n\ge 0$. Indeed, by induction on $n$ we can assume that $k_n\subseteq l'$, hence $t_n=\sum_{i=n}^\infty\pi_ia_i\in l'$ and then $l'=\wh{k_{n}(t_n)}$ contains $\pi_n a_n$ by Krasner's lemma, see \cite[Lemma~3.1.3(i)]{insepunif}. It follows that, in fact, $l'=l$. In particular, $l/k$ is primitive.
\end{proof}

\begin{cor}
If $k$ is the completion of a countable field then the extension $\whka/k$ is primitive. In particular, so are the extensions $\bfC_p/\bfQ_p$ and $\wh{\bfF_p((t))^a}/\bfF_p((t))$.
\end{cor}
\begin{proof}
The assumption on $k$ implies that $k^s/k$ is countably generated, hence $\whka=\wh{k^s}$ is a primitive extension of $k$ by Theorem~\ref{type1exam}.
\end{proof}

\begin{rem}\label{type1rem}
The geometric interpretation of the proof of Lemma~\ref{type1exam} is that one constructs a nested sequence of discs $E_n$ whose intersection is a single point $x$ of type 1 such that $l=\calH(x)$. As in \cite{berihes} we will denote the closed disc of radius $r$ and with center at $c$ by $E(c,r)=E_k(c,r)$. Then $E_n=E(c_n,r_n)$, where $c_n=\sum_{i=0}^{n-1}\pi_ia_i$ and $r_n=|\pi_na_n|$. Note that $E_n$ is defined over $k_n$ by Krasner's lemma, see \cite[Lemma~3.1.3(ii)]{insepunif}. So, $\cup_n k_n\subset\kappa(x)$ and this guarantees that $l\subseteq\calH(x)$.
\end{rem}

\subsubsection{Type 4 case: almost tame extensions}
Similarly to the case of extensions of type 1, we will construct examples using Krasner's lemma and the geometric interpretation is that we will find a sequence of nested discs $E_n=E(c_n,r_n)$ whose intersection will be a point of type 4. This case is subtler as we should guarantee that $\lim_n r_n=r>0$, and for this we will need a couple of technical lemmas that provide some control on the radii. We do not try to make the bounds as tight as possible. Set $|\kcirccirc|=\sup_{c\in\kcirccirc}|c|$. Thus, $|\kcirccirc|=1$ if the group $|k^\times|$ is dense, $|\kcirccirc|$ equals to the absolute value of the uniformizer if the valuation is discrete, and $|\kcirccirc|=0$ if the valuation is trivial.

\begin{lem}\label{disclem1}
Assume that $l/k$ is generated by an element $\alp\in\lcirc$ and let $r<r_\spl(\alp/k)|\kcirccirc|$. Then any $k$-split disc $E$ contains an $l$-split disc $E'$ whose radius in $E$ is larger than $r$.
\end{lem}
\begin{proof}
Fix $\pi\in\kcirccirc$ such that $r<r_\spl(\alp/k)|\pi|$. Choosing an appropriate coordinate on $E$ we can assume that $E=E(0,s)$, where $1\le s<|\pi|^{-1}$. Finally, choose $s'\in(r/|\pi|,r_\spl(\alp/k))$ and set $E'=E(\alp,s')$. Then $E'$ is $l$-split by Krasner's lemma, and the radius of $E'$ in $E$ is $s'/s>(r/|\pi|)|\pi|=r$, as required.
\end{proof}

\begin{lem}\label{disclem}
Assume that $l/k$ is a finite almost tame extension and let $r<|\kcirccirc|^3$. Then there exists a finite extension $l'/l$ such that any $k$-split disc $E$ contains an $l'$-split disc $E'$ whose radius in $E$ is larger than $r$.
\end{lem}
\begin{proof}
Set $p=\cha(\tilk)$. By the classical ramification theory, there exists a finite tame extension $l'/l$  such that $l'/k$ splits into the composition of $n$ extensions $k_{i+1}/k_i$, where $k_0=k$ and $k_n=l'$, of one of the following type: (a) unramified, (b) $k(a^{1/q})/k$ for a prime $q\neq p$ such that $|a|^{1/q}\notin|k^\times|$, (c) wildly ramified of degree $p$. We will prove that for any number $r<|k_i^{\circ\circ}|$ the extension $k_i/k_{i-1}$ is generated by an element $\alp_i\in k^\circ_i$ such that $r_\spl(\alp_i/k_{i-1})>r$. Assuming this claim, the assertion of the lemma follows by induction on $n$ with Lemma~\ref{disclem1} providing the induction step. This is obvious if the group $|k^\times|$ is dense, and if the valuation is discrete we use the simple fact that $\prod_{i=1}^n|k_i^{\circcirc}|\ge|\kcirccirc|$.

It remains to establish the claim. For shortness, we denote the extension $k_i/k_{i-1}$ by $L/K$. In case (a), the extension $\tilL/\tilK$ is separable, hence it is generated by a single element $\tilalp$. Any lifting $\alp\in\Lcirc$ of $\tilalp$ is a generator of $L/K$ with $r_\spl(\alp/K)=1$.

In case (b), we can replace $a$ by $a^nb^q$, where $1\le n\le q-1$ and $b\in K^\times$. In this way we can achieve that $a$ is a uniformizer in the discrete-valued case, and $|a|$ is arbitrarily close to $1$ if $|k^\times|$ is dense. In particular, we can achieve that $r^q<|a|<1$ and hence $\alp=a^{1/q}$ satisfies $r<|\alp|=r_\spl(\alp/K)$.

In case (c), the extension $L/K$ is immediate by \cite[Lemma~5.5.9]{Temkintopforms}. In particular, $|K^\times|$ is dense, and hence $L/K$ is almost unramified by \cite[Theorem~5.5.11]{Temkintopforms}. Pick any $x\in L\setminus K$. Then $L=K(x)$, and it suffices to prove that $r_\spl(x/K)=\inf_{c\in K}|x-c|$. Indeed, replacing $x$ by some $x-c$ we can then achieve that $r_\spl(x/K)>r|x|$, and, since $|K^\times|$ is dense, we can take $\alp=x/a$ for $a\in K$ with $|x|<|a|<r_\spl(x/K)/r$.

First, we claim that the infimum $s=\inf_{c\in K}|x-c|$ is not achieved. Indeed, if it is achieved for $c\in K$ then either $|x-c|\notin|K^\times|$ or $|x-c|=|a|$ for $a\in K$ and then $\tilb\notin\tilK$ for $b=(x-c)/a$. In any case, the extension $L/K$ is defectless and hence it is tame by \cite[Lemma~5.5.9]{Temkintopforms}, contradicting the assumption of (c).

Now, consider an affine line with a coordinate $t$ and let $z$ be the maximal point of $E=E(x,s)$ and $K'=\calH(z)$. The norm on $k[t]$ induced from $K'$ is the maximal norm such that $|x-c|=|t-c|$ for any $c\in K$. By Krasner's lemma, $E$ is not defined over $L$, and it follows easily that $L$ is not contained in $K'$. Thus, $L'=L\otimes_KK'$ is a field extension of $K'$ of degree $p$. Furthermore, $L'=K'(x)$ and $s=\inf_{c\in K'}|x-c|$ is achieved for $c=t$, in particular, $L'/K'$ is defectless. Since $$\Omega_{(L')^\circ/(K')^\circ}=\Omega_{\Lcirc/\Kcirc}\otimes_{\Kcirc}\Lcirc=0$$ \cite[Lemma~5.5.9]{Temkintopforms} implies that $L'/K'$ is tame and hence even unramified. Then the argument from (a) shows that $r_\spl(x/K')=s$, but clearly $r_\spl(x/K')=r_\spl(x/K)$.
\end{proof}

\begin{cor}\label{type4cor}
Assume that $k$ is a non-trivially valued analytic field and $l/k$ is a countably generated, infinite, almost tame, algebraic extension. Then there exists a primitive extension $K/k$ of type 4 such that $l\subset K$.
\end{cor}
\begin{proof}
Let $l=k(a_0,a_1,\dots)$ and set $l_n=k(a_0\.a_n)$. Then applying Lemma~\ref{disclem} inductively we can find a nested sequence of discs $E_0\supset E_1\supset\dots$ such that $E_n$ is $k_n$-split and the radius of $E_n$ in $E_{n-1}$ is bounded from below by a number close enough to $|k_{n-1}^{\circcirc}|^3$. In particular, for a fixed number $0<r<|\kcirccirc|^6$ we can choose $E_i$ such that the radii of $E_n$ in $E_0$ are bounded from below by $r$. For any point $x$ in the intersection $\cap_nE_n$ we have that $l\subset\kappa(x)$. Since $l/k$ is infinite, $x$ cannot be a $k^a$-point, and hence $\cap_nE_n$ is a single point $x$. Since the radii of $E_n$ do not tend to zero, $x$ is of type 4, and hence $K=\calH(x)$ is as required.
\end{proof}

\begin{rem}
The assumption that $l/k$ is infinite is essential for the construction. In fact, it is the assumption that guarantees that $k$ possesses extensions of type 4, in particular, $\whka$ is not spherically complete.
\end{rem}

\subsubsection{Type 4 case: deeply ramified extensions}
In the case of not almost tame extensions, the splitting radius can be small. Roughly speaking, the smaller the different is the smaller the splitting radius is. For example, $r=r_\spl(x/K)$ does not exceed $s=\inf_{c\in K}|x-c|$ for any $x\in L\setminus K$, and if $[L:K]=p$ and $K=K^t$ then $\Lcirc=\colim_i\Kcirc[a_ix+b_i]$, where $a_i,b_i\in\Kcirc$ and $|a_ix+b_i|\le 1$, by \cite[Proposition~6.3.13(i)(d)]{Gabber-Ramero}, and hence $\delta_{L/K}=(s/r)^{p-1}$ by \cite[Corollary~4.4.8]{radialization}.

Nevertheless, we are going to construct examples of type 4 extensions $K/k$ that contain algebraic subextensions $l/k$ of zero different. Moreover, we show that such examples are also ubiquitous.

\begin{theor}\label{deepramlem}
Assume that $k$ is a non-trivially valued analytic field with $p=\cha(\tilk)>0$ and $t\in\kcirc$ an element such that $d_{\kcirc}(t)$ is not infinitely divisible. Then there exists a primitive extension $K/k$ of type 4 and an algebraic subextension $l/k$ of $K/k$ such that $d_{\lcirc}(t)$ is infinitely divisible.
\end{theor}
\begin{proof}
Fix an increasing sequence of non-zero natural numbers $d_1,d_2\.$ converging to $\infty$. Replacing $t$ with $p^{-n}t$ in the mixed characteristic case we can assume that $|p|<|t|\le 1$. Also, fix $\pi\in\kcirccirc$ such that $|p|\le |\pi|$. Now, we inductively set $t_0=t$, $t_{n+1}$ is a root of $f_n(x)=x^{p^{d_n}}-\pi x-t_n$, $k_0=k$, $k_{n+1}=k_n(t_{n+1})$, and $l=\cup_n k_n$. Note that $dt_n=a_ndt_{n+1}$, where $a_n=p^{d_n}t_{n+1}^{p^{d_n}-1}-\pi\in k_{n+1}$ satisfies $|a_n|=|\pi|$. Therefore, each $dt_n$, including $dt$, is infinitely divisible in $\Omega_{\lcirc}$. In particular, the different of the extension $l/k$ is zero, and hence $[l:k]=\infty$ by \cite[Theorem~5.2.11(iii)]{Temkintopforms}.

Next, let us compute $r_\spl(t_{n+1}/k_{n})$. If $\cha(k)=p$ then the differences between $t_{n+1}$ and its conjugates are of the form $\pi^{1/(p^{d_n}-1)}$, hence $r_\spl(t_{n+1}/k_{n})=|\pi|^{1/(p^{d_n}-1)}$. Moreover, inspecting the situation modulo $p$, one obtains that the same formula for $r_\spl(t_{n+1}/k_{n})$ holds in the mixed characteristic case as well. Since $d_n$ tend to infinity, we obtain that $\prod_{n=1}^\infty r_\spl(t_{n+1}/k_{n})>0$, and then using Lemma~\ref{disclem1} one constructs a nested sequence of discs $E_n$ such that $E_n$ is $k_n$-split and the radii of $E_n$ tend to a positive number. By the same argument as in the proof of Corollary~\ref{type4cor}, $\cap_nE_n$ is a single point $x$ of type 4 and $K=\calH(x)$ contains $l$.
\end{proof}

\begin{rem}\label{deepramrem}
If $d_{\kcirc}(t)$ is not infinitely divisible, one can also construct examples of extensions $K/k$ of type 4 such that $k$ is algebraically closed in $K$ and $d_{\Kcirc}(t)$ is infinitely divisible. If $\cha(k)=p$, such an example is constructed in the proof of \cite[Theorem~1.2]{untilt}.
\end{rem}

\begin{rem}
Assume that $l/k$ is a countably generated, infinite, separable extension. Similarly to Theorem~\ref{type1exam} one may wonder whether any such $l/k$ embeds into a primitive extension $K/k$ of type 4. Probably this is true. At least, by Corollary~\ref{type4cor} and Theorem~\ref{deepramlem} this is so for a large class of extensions.
\end{rem}

\subsubsection{Spherically complete fields}
Our results on extensions of type 4 imply a somewhat surprising consequence on the structure of spherically complete fields.

\begin{theor}\label{sphth}
If $\whka$ is spherically complete then $\whka=k^a$ and $[k^a:k]\le 2$.
\end{theor}
\begin{proof}
If $k$ is not almost tame then it is not deeply ramified by \cite[Theorem~5.5.15]{Temkintopforms}], and hence $\Omega_{\kcirc}$ is not infinitely divisible by the equivalence (i)$\Longleftrightarrow$(v) in \cite[Proposition~6.6.6]{Gabber-Ramero}. Thus, $k$ possesses a primitive extension $K$ of type 4 by Theorem~\ref{deepramlem}. In particular, $\wh{k^aK}/\whka$ is an extension of type 4, and hence $\whka$ is not spherically complete.

Assume now that $k$ is almost tame. Then the almost tame extension $k^a/k$ is finite because otherwise Corollary~\ref{type4cor} would imply that $k$ possesses extensions of type 4, which, as we shown above, is impossible. Therefore, $k^a$ is already complete and $[k^a:k]\le 2$ by a theorem of E. Artin.
\end{proof}

\subsection{Non-additivity of topological transcendence degree}

\subsubsection{Composition of two primitive extensions}
Consider now a tower $l=\wh{k(x)}$ and $L=\wh{l(y)}$ of two primitive extensions of type different from 1, and assume for simplicity that $k=k^a$. Then $L$ is a primitive extension of $K=\wh{k(y)}$, and one may naturally expect that $L/K$ is topologically transcendental. If this happens then the set $\{x,y\}$ is topgebraically independent over $k$ and hence $\topdeg(L/k)=2$, as one might expect. However, if $\cha(\tilk)>0$ then the pathological situation with $L\subseteq\whKa$ can happen. In particular, the topgebraic dependency turns out to be asymmetric: $x\in\wh{k(y)^a}$ but $y\notin\wh{k(x)^a}$.

\begin{theor}\label{mainpathology}
Assume that $k$ is an analytic field of positive residual characteristic and $l=\wh{k(x)}$ is a topologically transcendental primitive extension. Then $l$ possesses a primitive extension $L=\wh{l(y)}$ of type 4 such that $L\subseteq\whKa$ for $K=\wh{k(y)}$ and $|x-y|<r_k(x)$. In particular, $K$ is $k$-isomorphic to $l$, the extensions $L/l$, $l/k$ and $L/k$ are of topological transcendence degree one, and the topological transcendence degree is not additive for the tower $L/l/k$.
\end{theor}
\begin{proof}
By \S\ref{difsec}, $\whd_{l/k}(x)\neq 0$ and hence $d_{\lcirc/\kcirc}(x)$ is not infinitely divisible by Lemma~\ref{kahlerlem}. In particular, $d_{\lcirc}(x)$ is not infinitely divisible. By Theorem~\ref{deepramlem} there exists an extension $L=\wh{l(y)}$ of type 4 such that $d_{\Lcirc}(x)$ is infinitely divisible. Then $d_{\Lcirc/\Kcirc}(x)$ is infinitely divisible, and hence $\whd_{L/K}(x)=0$ by Lemma~\ref{kahlerlem}. In particular, $L/K$ is of type 1 by \S\ref{difsec}, that is, $L\subseteq\whKa=\wh{k(y)^a}$.

The assertions about the topological degrees are clear. It remains to note that we can replace $y$ by any element $z\in k(y)^a$ transcendental over $k$, hence we can achieve that $|y-x|<r_k(x)$ and then $K\toisom l$ by Lemma~\ref{primdeform}.
\end{proof}

\begin{rem}
(i) The critical input here is an example of a primitive extension $L=\wh{l(y)}$ which makes $dx$ infinitely divisible in $\Omega_{\Lcirc}$ and hence kills $\whd x$ in $\whOmega_L$. In particular, $\Omega_{\Lcirc/\lcirc}$ should have infinitely divisible torsion elements and $L/l$ should have zero different.

(ii) This theorem extends \cite[Theorem~1.2]{untilt}, but the argument is similar. The novelty is that the construction of a ``very wildly ramified" extension $L/l$ in this paper is based on Theorem~\ref{deepramlem} and applies to the most general case.
\end{rem}

\subsubsection{$k$-endomorphisms of $\wh{k(x)^a}$}\label{endsec}
Now we can reprove the theorem of Matignon-Reversat, \cite[Th\`eor\'eme~2]{matignon-reversat}.

\begin{theor}\label{endth}
Assume that $k$ is an analytic field of positive residual characteristic and $K=\wh{k(x)^a}$ is a topologically transcendental extension of $k$. Then there exists a non-invertible $k$-endomorphism of $K$ which is a quasi-perturbation of $\Id_K$.
\end{theor}
\begin{proof}
It suffices to construct a $\whka$-endomorphism, so we can assume that $k=k^a$. By Theorem~\ref{mainpathology}, $l=\wh{k(x)}$ possesses an extension $\wh{l(y)}$ of type 4 such that $L=\wh{l(y)^a}$ equals to $\wh{k(y)^a}$ and there is a perturbation of the identity $\phi_0\:\wh{k(y)}\toisom\wh{k(x)}$ sending $y$ to $x$. Clearly, $\phi_0$ extends to topgebraic closures providing an endomorphism $\phi\:L\to K$. The check that $\phi$ is a quasi-perturbation of the identity is easy, and we skip it. Since $K\subsetneq L$, the restriction of $\phi$ onto $K$ is a required non-invertible $k$-endomorphism of $K$.
\end{proof}

\begin{rem}\label{endrem2}
(i) We deduced the theorem of Matignon-Reversat from Theorem~\ref{mainpathology}, but, in fact, they are equivalent. Indeed, if $\phi$ is a non-invertible endomorphism of $\wh{k(y)^a}$ then for $x=\phi(y)$ we have that $x\in\wh{k(y)^a}$ but $y\notin\wh{k(x)^a}$. Note that $\{x\}$ is a maximal topgebraically independent set in $\wh{k(y)^a}$, which is not topgebraically generating.

(ii) Another corollary of Theorem~\ref{endth} is the following observation of Matignon and Reversat: $\wh{k(x)}$ contains infinitely many distinct algebraically closed analytic extensions of $k$. For example, one can take $K_n=\phi^n(K)$.

(iii) The argument in the theorem implies that for any element $y\in\wh{k(x)^a}$ with $|x-y|<\alp r_k(x)$, where $\alp\in(0,1)$, there exists a unique isomorphism $\phi_y\:\wh{k(x)^a}\toisom\wh{k(y)^a}$ taking $x$ to $y$, and this $\phi_y$ is a quasi-perturbation of the identity. We know from Remark~\ref{deformrem}(i) that if $\phi_y$ is non-invertible then it is not a perturbation of the identity. In fact, if $\cha(\tilk)>0$ then $\phi_y$ is never a perturbation of the identity because $$\lim_{n\to\infty}\left|x^{1/p^n}-y^{1/p^n}\right|/\left|x^{1/p^n}\right|=1.$$ On the other hand, it is easy to see that $\phi_y|_L$ is an $\alp$-perturbation of $\Id_L$ for any tame (even almost tame) algebraic extension $L/\wh{k(x)}$, regardless of $p=\cha(\tilk)$. In particular, if $p=0$ then $\phi_y$ is an $\alp$-perturbation of the identity.
\end{rem}

\subsubsection{Naive multitype}\label{multitype}
Theorem \ref{mainpathology} can be translated to Berkovich spaces as follows. Consider the $k$-analytic affine plane $\bfA^2_k$ with coordinates $x,y$ and let $z$ be the point corresponding to the norm that $L$ induces on $k(x,y)$. In particular, $L=\calH(z)$. The projections $z_1,z_2$ of $z$ onto the affine lines have completed residue fields $\calH(z_1)=l$ and $\calH(z_2)=K$. They are of the same type over $k$, which can be 2, 3 or 4. On the other hand, the types of $z$ in the fibers correspond to the types of the extensions $L/l$ and $L/K$, which are 4 and 1, respectively. In particular, the naive definition of multitype of $z$ (see \S\ref{multisec}) is incorrect.

\subsubsection{Multitype}
The correct notion of multitype is introduced in Definition~\ref{multirem}. Given a $k$-analytic space $X$ and a point $x\in X$ with $\dim_x(X)=d$, by a {\em multifibration} at $x$ we mean an analytic domain $X_d\subseteq X$ containing $x=x_d$ and a sequence of maps $X_d\to X_{d-1}\to\dots\to X_0$ such that $X_i$ is of dimension $i$ at the image $x_i$ of $x$ and the fiber of $X_{i+1}\to X_i$ over $x_i$ is a curve. It is easy to see that multifibrations at $x$ always exist: for a small enough $X_d$ one can find a morphism $X_d\to\bfA^d$ with finite fibers, and then $X_i=\bfA^i$ for $i<d$ will work.

\begin{theor}\label{multith}
Assume that $X$ is a $k$-analytic space and $x\in X$ is a point. Then,

(i) For any mutlifibration $X_d\to X_{d-1}\to\dots$, at least $n_4(x)$ points from the set $\{x_1\.x_d\}$ are of type 4 in the fiber.

(ii) There exists a multifibration such that precisely $n_4(x)$ points are of type 4.

(iii) If $\cha(\tilk)=0$ then the number of points of type 4 is $n_4(x)$ for any multifibration.
\end{theor}
\begin{proof}
First we note that for any $K$-analytic space $Y$ with a point $y\in Y$ the extension $\calH(y)/K$ is topologically finitely generated. In particular, it possesses a topological transcendence basis and satisfies the equality $\topdeg(\calH(y)/K)=\Topdeg(\calH(y)/K)$ by Theorem~\ref{basisth}. All extensions considered below are of this form, so we do not have to distinguish these two invariants, and then $\topdeg$ is subadditive by Lemma~\ref{subadditivelem}.

The invariants $\trdeg(K/k)$ and $\dim_\bfQ(|K^\times|/|k^\times|\otimes\bfQ)$ are additive in towers, hence the number of extensions of type 2 and 3 is $n_2(x)$ and $n_3(x)$, respectively. Furthermore, we always have that the number of points of type 2, 3 or 4 is at least $\topdeg(\calH(x)/k)$ by the subadditivity of $\topdeg$. This proves (i), and since $\topdeg$ is additive when the residue characteristic is zero, we also obtain (iii).

Let us prove (ii). Shrinking $X$ around $x$ we can assume that it is affinoid. Fix a topological transcendence basis $S$ of $\calH(x)/k$, and choose a perturbation $S'\subset\kappa(x)$ of $S$. Then $S'$ is topgebraically independent by Theorem~\ref{deformth} and hence a topological transcendence basis by Theorem~\ref{basisth}. So, replacing $S$ by $S'$ we can assume that it lies in $\kappa(x)$ and hence lifts to a subset $f_1\.f_n\subset\calO_{X,x}$, where $n=\topdeg(\calH(x)/k)$. Shrinking $X$ again we can assume that $f_1\.f_n$ are global functions and hence define a morphism $f\:X\to Y=\bfA^n_k$. Shrinking $X$ further we can easily extend $f$ to a multifibration $X\to X_{d-1}\to\dots$ such that $X_n=Y$. By the construction, $\calH(x)$ is topgebraic over $\calH(x_n)=\wh{k(S)}$, and hence the points $x_d,x_{d-1}\. x_{n+1}$ are of type 1 in the fibers. Thus, the points $x_n\.x_1$ are precisely the points of types 2, 3 and 4, as asserted.
\end{proof}

\subsection{Non-existence of a topological transcendence basis}\label{nobasissec}

\subsubsection{Spherically complete extensions}
Here is our first example of an extension that does not admit such a basis.

\begin{theor}\label{sphcor}
Assume that $K/k$ is an extension of analytic fields  such that $\whka$ is not spherically complete and $K$ is algebraically closed and spherically complete. If $k$ is of positive residual characteristic then the extension $K/k$ does not admit a topological transcendence basis.
\end{theor}
\begin{proof}
Assume, to the contrary, that $S$ is a topological transcendence basis. Choose an element $t\in S$ and set $l=\wh{k(S\setminus\{t\})}$. Then $L=\wh{l(t)}$ is a primitive extension of $l$ of type different from 1 and $K=\whLa$. By \S\ref{difsec}, $\whd_{L/l}(t)\neq 0$, hence $d_{\Lcirc/\lcirc}(t)$ is not infinitely divisible by Lemma~\ref{kahlerlem}. So, $L$ possesses a primitive extension $F$ of type 4 by Theorem~\ref{deepramlem}, and then $\wh{L^aF}/K$ is an extension of type 4, and hence $K$ is not spherically complete. A contradiction.
\end{proof}

\begin{question}
What happens if $\cha(\tilk)=0$? For example, what happens for spherically complete extensions of $\bfC((t))$?
\end{question}

\subsubsection{Transcendence degree one}
One can show that the topological transcendence degree in the previous example is infinite. In the following example it is finite.

\begin{theor}\label{nobasisth}
Assume that $k$ is an analytic field of positive residual characteristic. Then the class of extensions $K_i$ of $k$ such that $\topdeg(K_i/k)=1$ contains maximal objects with respect to inclusion, and if $K/k$ is such a maximal extension then $K/k$ does not possess a minimal topgebraically generating set. In particular, $\Topdeg(K/k)$ is infinite and $K/k$ does not admit a topological transcendence basis.
\end{theor}
\begin{proof}
Assume that $\{K_i\}_{i\in I}$ is a set of extensions of $k$ of topological transcendence degree one which is totally ordered with respect to inclusion. Then $\cup_{i\in I}K_i$ is a valued field and we claim that its completion $K$ is of topological transcendence degree 1. Indeed, if a set $\{x,y\}\subset K$ is topgebraically independent over $k$ then by Theorem~\ref{deformth} the same is true for any perturbation $\{x',y'\}$. By the density of $\cup_{i\in I}K_i$, we can achieve that $\{x',y'\}\subset K_i$ for some $i$, but this would contradict that $\topdeg(K_i/k)=1$.

The existence of a maximal $K$ would now follow from Zorn's lemma once one guarantees that $K_i$ can be chosen to be subsets of a fixed universe set. In other words, we need to establish a uniform bound on the cardinalities of $K_i$. For expository reasons we postpone this until Corollary~\ref{cardcor}.

Finally, let $S$ be a topgebraically generating set. First, we claim that $|S|>1$. Indeed, otherwise $K=\wh{k(x)^a}$ and then $K$ is $k$-isomorphic to its proper subfield by Theorem~\ref{endth}. Therefore, there exists a non-trivial extension $L/K$ such that $L$ is $k$-isomorphic to $K$, and this contradicts the maximality of $K$. Next, we note that for any pair of distinct elements $a,b\in S$, we can remove one of them from $S$ without loosing topgebraic generation. Indeed, $\topdeg(K/k)=1$ and hence either $a$ is topgebraic over $\wh{k(b)}$ and we can remove $a$, or $b$ is topgebraic over $\wh{k(a)}$ and we can remove $b$. All in all, the generating set $S$ is not minimal.
\end{proof}

\begin{rem}
Both examples are based on Zorn's lemma. In a sense they indicate that large extensions tend not to have a topological transcendence basis. One may still wonder if there exists a constructive (e.g. countable) construction of such an extension. We do not know the answer, but this naturally leads to the questions below.
\end{rem}

\begin{question}\label{topdeg1q}
Assume that $\cha(\tilk)>0$ and $K=\wh{k(x)^a}\neq\whka$. Let $\phi\:K\to K$ be a non-invertible $k$-endomorphism. Set $L=\wh{K_{-\infty}}$, where informally $K_{-\infty}=\cup_{n\in\bfN}\phi^{-n}(K)$. More precisely, $K_{-\infty}$ is the colimit of the diagram consisting of $K$ and the endomorphisms $\{\phi^n\}_{n\in\bfN}$, and on the practical level $K$ is the union of $\bfN$ copies $K_{-n}=\wh{k(x_n)}$ of $K$ and under the identification $K=K_n$ one has that $\phi(x_n)=x_{n-1}$. Is $\Topdeg(L/k)$ infinite? Does the answer depend on $\phi$? Probably, this is tightly related to the structure of the totally ordered set $\End_k(K)/\Aut_k(K)$, see Remark~\ref{subfieldsrem} below.
\end{question}

\subsubsection{The set $\End_k(K)/\Aut_k(K)$}
Assume that $\cha(\tilk)>0$ and $K=\wh{k(x)^a}$ with $x\notin\whka$. By the theorem of Matignon-Reversat, $K$ contains infinitely many distinct subfields of the form $\wh{k(t)^a}$. However, the fact that $\topdeg(K/k)=1$ by Theorem~\ref{mainineq} imposes a strong restriction on the set of such subfields. Slightly more generally, we have the following result.

\begin{theor}\label{subfieldsth}
Assume that $K/k$ is an extension of analytic fields such that $\topdeg(K/k)\le 1$. Then the set $S_{K/k}$ of all subfields of $K$ of the form $\wh{k(t)^a}$ is totally ordered with respect to inclusion.
\end{theor}
\begin{proof}
Consider two subfields $\wh{k(x)^a}$ and $\wh{k(y)^a}$ of $K$ and assume that neither of them contains the other one. Then $\{x,y\}$ is a topgebraically independent set, contradicting that $\topdeg(K/k)=1$.
\end{proof}

\begin{question}\label{subfieldsq}
What is the structure of the totally ordered set $S_{K/k}$?
\end{question}

\begin{rem}\label{subfieldsrem}
If $K=\wh{k(t)^a}$ then $S_{K/k}=\End_k(K)/\Aut_k(K)$ and the projection $q\:\End_k(K)\to S_{K/k}$ associates to $\phi\in\End_k(K)$ the subfield $\phi(K)$. Thus, $q$ can be viewed as an invariant of endomorphisms that measures how far $\phi$ is from being invertible. However, to make this abstract invariant a useful one we should know something about the structure of $S_{K/k}$.
\end{rem}

\begin{cor}\label{cardcor}
Assume that $K/k$ is an extension of analytic fields such that $\topdeg(K/k)=1$. Let $\kappa$ be the cardinality of $k$ and $\omega=\max(\kappa,2^{\aleph_0})$. Then the cardinality of $K$ does not exceed ${2^\omega}$.
\end{cor}
\begin{proof}
Replacing $K$ by $\whKa$ we can assume that $K$ is algebraically closed. Let $\{K_j\}_{j\in J}$ denote the totally ordered family of all subfields of $K$ of the form $K_j=\wh{k(x_j)^a}$. Note that $k(x)^a$ is of cardinality $\rho=\max(\kappa,\aleph_0)$ and hence $\card(K_j)\le\rho^{\aleph_0}=\omega$. The cardinality of any cut $J_{<j}=\{j'\in J|\ j'<j\}$ does not exceed $\omega$ because sending $j'$ to $x_{j'}$ one obtains an embedding $J_{<j}\into K_j$. Therefore $\card(J)\le 2^\omega$, and since $K=\cup_{j\in J} K_j$ we obtain that $\card(K)\le\omega2^\omega=2^\omega$.
\end{proof}

In fact, the above bound can be improved using Poonen's results on maximal (or spherical) completions. We briefly indicate the argument.

\begin{rem}
(i) The case when $\kappa<2^{\aleph_0}$ can only happen when the valuation of $k$ is trivial. In the sequel, assume that $|k^\times|\neq 1$ and so $\omega=\kappa$.

(ii) Note that $K/K_j$ is immediate and so $\card(K)$ is bounded by the cardinality of the maximal completion $L$ of $K_j$, see \cite[Therem~2]{Poonen}. It follows from the description of $L$ as a Mal'cev-Neumann field in \cite[Section~4]{Poonen} that $$\card(L)=\card(|K_j^\times|)\card(\tilK_j)^{\aleph_0}=\max(2^{\aleph_0},\card(\tilK_j))\le\kappa.$$
Hence $\card(L)=\card(k)=\kappa$.
\end{rem}

\bibliographystyle{amsalpha}
\bibliography{pathologies}

\end{document}